\documentclass[12pt, oneside,reqno]{amsproc}   	% use "amsart" instead of "article" for AMSLaTeX format
\usepackage{geometry}                		% See geometry.pdf to learn the layout options. There are lots.
\usepackage{graphicx}				% Use pdf, png, jpg, or eps with pdflatex; use eps in DVI mode
\usepackage{amssymb}
\usepackage{amsmath}
\usepackage[all]{xy}

\newtheorem{thm}{Theorem}[section]
\newtheorem{lem}[thm]{Lemma}

\newtheorem{prop}[thm]{Proposition}

\numberwithin{equation}{section}

\def\Pb{\ifmmode{\Bbb P}\else{$\Bbb P$}\fi}
\def\Z{\ifmmode{\Bbb Z}\else{$\Bbb Z$}\fi}
\def\C{\ifmmode{\Bbb C}\else{$\Bbb C$}\fi}
\def\R{\ifmmode{\Bbb R}\else{$\Bbb R$}\fi}
\def\S{\ifmmode{S^2}\else{$S^2$}\fi}

\def\S{\cal S}

\begin{document}
	
\title[Asymptotics]{Precise asymptotics near a generic $\mathbb S^1\times\mathbb R^3$ 
singularity of mean curvature flow}

\begin{abstract} In the present paper we study a type of generic singularity of mean curvature flow modelled on the bubble-sheet $\mathbb S^1\times\mathbb R^3$ , and we derive an asymptotic profile for a neighborhood of singularity. 
\end{abstract}

\author{Zhou Gang$^\star$\footnote{$^\star$ partly supported by Simon collaborative grant 709542}}
\address{Department of mathematics and statistics\\
	Binghamton University}
	\email{gangzhou@binghamton.edu}
\author{Shengwen Wang} 
\address{School of Mathematical Sciences\\
	Queen Mary University of London}
	\email{shengwen.wang@qmul.ac.uk}

\maketitle

\section{Introduction}
We consider a family of hypersurfaces $\{M^n_t\subset\mathbb R^{n+1}\}$ evolving by mean curvature flow (MCF), which makes it satisfy the equation
\begin{align}
\frac{d\mathbf X_M}{dt}^\perp=\mathbf H_M.
\end{align} where $\mathbf X_{M}$ is the position vector and $\mathbf H_M$ is the mean curvature vector. For a rich family of initial hypersurfaces (in particular all closed hypersurfaces), singularities will form in a finite time. A main theme in the study of MCF or more general geometric flows is to understand the formation of singularities and find a way to extend the flows through singularities with controlled geometry and topology. 

According to Huisken's monotonicity formula \cite{H}, the asymptotic behaviour of mean curvature flow near a singularity is modelled on self shrinkers (possibly with multiplicities). Self shrinkers are hypersurfaces satisfying the equation $$\mathbf H_M+\frac{1}{2}\mathbf X_M^\perp=0,$$ and by the result of Colding-Minicozzi \cite{CM}, the generic singularities must be modelled on mean convex self shrinkers $\mathbb S^k_{\sqrt{2k}}\times\mathbb R^{n-k}$ for $1\leq k\leq n-1$. When the evolving hypersurfaces are 2-convex, there is a well developed theory for mean curvature flow with surgeries \cite{HK, HS}, with a detailed description near a neighborhood of the singularities and surgery procedure with controlled geometry and topology. For general hypersurfaces, the singularity analysis is hard. However, we note that there have been some breakthroughs in singularity analysis in the past few years, especially in dimension 3. By \cite{CHH, CHHW}, any neck singularities, i.e. modeled on $\mathbb S^{n-1}\times\mathbb R$, must have a mean convex neighbourhood in space-time. Based on this, one can construct smooth MCF with surgery if all the singularities are neck-type by \cite{D}. For generic initial data, \cite{CCMS1, CCMS2} showed that the flow only develop cylindrical or spherical singularities. More recently, Bamler-Kleiner \cite{BK} verified Ilmanen's conjecture \cite{I2} that the tangent flows at singularities always appear as a multiplicity 1 self shrinker in dimension 3. Combining previous breakthroughs and Brendle's classification of genus zero shrinkers, we have the well-posedness of mean curvature flow from topological spheres in $\mathbb R^3$. We also note that when the ambient dimension is 4 and entropy is small, there are also recent results on controlling the topological changes across singularities \cite{BW, BWang, CCMS1, CCMS2, MW}. In the present paper we would like to give a precise asymptotic for a class of generic bubble-sheet singularities modelled on $\mathbb S^1\times\mathbb R^3$ at the singular time in $\mathbb R^5$. In general, there is no canonical neighborhood theorem known for  bubble-sheet type singularities $\mathbb S^k\times\mathbb R^{n-k}$ with $k>1$ and no theory of performing surgery through these singularities. Recently, the first author \cite{GZ} gives a description in a certain regime of a neighborhood of the bubble-sheet singularities modeled on $\mathbb S^1\times \mathbb R^3$. We will use that result to derive a precise asymptotic of a generic singularity modeled on $\mathbb S^1\times\mathbb R^3$.  

%In the situation of cylindrical singularities, the phenomena that the mean curvature evolution collapse to the singular set in the first and only singular time is highly unstable and can only appear in very rare cases: evolution of exact round cylinders collapsing to a line and the evolution of rotational symmetric marriage rings collapsing to a circle for example. Generically, the flow encountering cylindrical singularities will develop pinched-necks (if the axis of cylinder is one-dimensional) or pinched bubble-sheet (if the axis of cylinder is higher-dimensional). Our result concerns a bubble-sheet singularity modeled on $\mathbb S^1\times\mathbb R^3$.

%The paper is organized as follows: the main theorem will be stated in Section \ref{sec:Main}, and proved in Section \ref{sec:ProMain}. In Section \ref{sec:KeyLem} we will prove a key technical lemma.

%\section{Main Theorem}\label{sec:Main}
We start with formulating the problem. In this paper, we consider MCF of four dimensional hypersurfaces $M^4_t$ in $\mathbb R^5$ that develops a generic singularity modeled on the self-shrinking cylinder $\mathbb S^1_{\sqrt{2}}\times\mathbb R^3$ at the space-time point $(0,0)$. The result of Colding and Minicozzi \cite{CM1} implies the uniqueness of a fixed axis. Up to a rotation, we can assume the unique singularity model is the cylinder 
\begin{align}
C_1=\Big\{\begin{bmatrix}y\\ \sqrt2\cos\theta \\ \sqrt2\sin\theta\end{bmatrix}\in\mathbb R^5\ \Big|\ y\in\mathbb R^3,\ \theta\in [0,2\pi)\Big\}.
\end{align}
In a neighborhood of the blowup point MCF can also be parametrized by 
\begin{align}
M_t = \begin{bmatrix}x\\ u(x,\theta,t)\cos\theta \\ u(x,\theta,t)\sin\theta \end{bmatrix},\ x\in U_t\subset\mathbb R^3,\ \theta\in [0,2\pi)\label{unscaled}
\end{align}

Consider a time reparametrization $\tau:=-\ln|t|$ and define a rescaled MCF by $$\tilde M_\tau=\sqrt{-t}M_t.$$ 
Hence the time $\tau$ slice of rescaled flow $\tilde{M}_\tau$ can be locally parametrized by 
\begin{align}
\tilde{M}_\tau = \begin{bmatrix}y\\ v(y,\theta,\tau)\cos\theta \\ v(y,\theta,\tau)\sin\theta \end{bmatrix}, \ y\in U_\tau\subset\mathbb R^3, \  \theta\in [0,2\pi), 
\end{align} where the function $v$ is defined in terms of $u$ by
\begin{align}\label{eq:uvRe}
v(y,\theta,\tau)=\frac{1}{\sqrt{-t}}u(x,\theta,t),
\end{align}
with $y\in \mathbb{R}^3$  defined as
\begin{align}
 y:=\frac{x}{\sqrt{-t}}.
\end{align}
Under our assumption, this rescaled MCF converges (without passing to a subsequence) to the unique cylinder $C_1$ as measures on any compact subset, see \cite{CM1}.

Before stating our main theorem, we would like to record the following result of the first author \cite{GZ} on the asymptotics before the singular time of a generic singularity modelled on $\mathbb S^1\times\mathbb R^3$.

\begin{thm}\label{GZthm}[Theorem 2.1 of \cite{GZ}] Let $M^4_t$ be a MCF in $\mathbb R^5$ that develops a generic singularity at the space-time $(0,0)\subset\mathbb R^5\times\mathbb R$ and is modeled on the self-shrinker $C_1$ defined above.

Then there exist two independent positive constants $\tau_0\gg1$ and $C$ such that for $\tau\geq\tau_0$,  $v$ has the following asymptotic
\begin{equation}\label{asym}
v(y,\theta,\tau)=\sqrt{\frac{2+y^T\cdot B(\tau)\cdot y}{2a(\tau)}}+\eta(y,\theta,\tau),\ \text{for}\ |y|\leq \Omega(\tau),
\end{equation}
where $y\in\mathbb R^3$ and $\theta\in [0,2\pi)$, $B$ is a $3\times 3$ symmetric real matrix satisfying
\begin{align}\label{eq:b123}
B=\tau^{-1}\left[
\begin{array}{ccc}
b_1&0&0\\
0&b_2&0\\
0&0&b_3
\end{array}\right]+O(\tau^{-2}),\ b_k=0\ \text{or}\ 1,\  \text{and}\ k=1,2,3,
\end{align} $a$ is a scalar function satisfying
\begin{align}\label{eq:aaa}
a(\tau)=\frac{1}{2}+O(\tau^{-1})
\end{align}
and $\eta$ satisfies the estimate
\begin{align}\label{eq:InftyEst}
\begin{split}
\|\left(\sqrt{1+y^2}\right)^{-3}1_{\Omega}\partial_{\theta}^{l}\eta(\cdot,\tau)\|_{\infty}\leq &C \tau^{-2},\ l=0,1,2,\\
\|\left(\sqrt{1+y^2}\right)^{-2}1_{\Omega}\partial_{\theta}^{l}\nabla_{y}\eta(\cdot,\tau)\|_{\infty}\leq &C \Omega^{-3},\ l=0,1,\\
\|\left(\sqrt{1+y^2}\right)^{-1}1_{\Omega}\nabla_{y}^{k}\eta(\cdot,\tau)\|_{\infty}\leq &C \Omega^{-2},\ |k|=2.
\end{split}
\end{align}

\end{thm}

Here $1_{\leq \Omega}$ is the Heaviside function defined as
\begin{align}
1_{\leq \Omega}(y)=\left[
\begin{array}{ll}
1\ &\ \text{if}\ |y|\leq \Omega,\\
0\ & \ \text{otherwise.}
\end{array}
\right.
\end{align}
and the scalar function $\Omega$ is defined as
\begin{align}
\Omega(\tau):=\sqrt{100 \ln \tau+9 (\tau-\tau_0)^{\frac{11}{10}}}.\label{eq:defOmega}
\end{align}

Under a further generic condition, specific in \eqref{eq:b123},
\begin{equation}\label{generic}
b_1=b_2=b_3=1,
\end{equation}
it is shown in \cite{GZ} Theorem 2.3 that MCF develops an isolated singularity of bubble-sheet type verifying Ilmanen's mean-convex neighborhood conjecture in this case. We would like to give a more precise asymptote near such singularities, and hopefully this will help the development of a surgery theory for bubble-sheet type singularities.

The main theorem is
\begin{thm}\label{main}
Let $\{M^4_t\subset\mathbb R^5\}$ be an evolution of MCF. Assume that it develops a singularity at space-time $(0,0)$, modeled on a cylindrical self-shrinker 
\begin{align}
\mathbb{S}^{1}_{\sqrt{2}} \times \mathbb{R}^3=\left(
\begin{array}{cc}
x\\
\sqrt2\cos(\theta)\\
\sqrt2\sin(\theta)
\end{array}
\right),\ \text{for}\ x\in \mathbb{R}^3.
\end{align} 

Under the generic assumption \eqref{generic}, in a neighborhood $B_\epsilon(0)\subset\mathbb R^5$ of the origin, and at the singular time $t=0$, $M_0$ can be parametrized by\begin{align}
\left(
\begin{array}{cc}
x\\
u_0(x,\theta)\cos(\theta)\\
u_0(x,\theta)\sin(\theta)
\end{array}\right)
\in\mathbb{R}^5,\ x\in\mathbb{R}^3,\ \theta\in [0,2\pi),
\end{align} where $u_0$ is some nonnegative function,
and satisfies the asymptotics 
\begin{align}
u_0(x,\theta)=\frac{|x|}{\sqrt{-2\ln|x|}}(1+o(1)), \ \text{where}\ o(1)\rightarrow 0\ \text{as}\ |x|\rightarrow 0.
\end{align}
\end{thm}

For the known works, in \cite{AK}, Angenent and Knopf has obtained a similar asymptotic for rotationally symmetric Ricci flow neck-pinch singularities.

\textbf{Acknowledgements.}
We would also like to thank the referee for the comments and suggestions that help improve our manuscript.

\section{Extension of the asymptotic profile to the singular time}\label{sec:ProMain}
To prepare for the proof we need a technical result, specifically, Lemma \ref{keylemma} below. It tells us that the asymptotic profile in Theorem \ref{GZthm}, which holds in a shrinking neighbourhood for the unscaled flow towards the singular time, can be extended up to the singular time. 

Before stating it we define some constants. For any fixed $x_1\in \mathbb{R}^3$, with $0<|x_1|\ll 1$, define $\tau_1$ to be the large number such that 
\begin{align}
e^\frac{\tau_1}{2}|x_1|=\tau_1^{\frac{1}{2}+\frac{1}{20}}.\label{eq:defX1}
\end{align} It is not hard to see that such a number uniquely exists for $|x_1|$ small enough since $e^{\frac{\tau_1}{2}}$ grows much faster than $\tau_1^{\frac{1}{2}+\frac{1}{20}}$ as $\tau_{1}$ grows.
We denote, by $t_1$, the time for MCF corresponding to the rescaled time $\tau_1$, which makes
\begin{align}
t_1=t_1(\tau_1)=-e^{-\tau_1}.\label{eq:twoScal}
\end{align}

The following result is important for proving Theorem \ref{main}.
\begin{lem}\label{keylemma}
Suppose that all the conditions in Theorem \ref{main} hold. 

For any $\epsilon>0$, there exits a small constant $\delta>0$ such that if $0< |x_1|<\delta$, then the function $u(x,\theta,t)$ in \eqref{unscaled} satisfies the following estimate at $x=x_1$: when $t\in [t_1,0)$, with $t_1$ defined in \eqref{eq:defX1} and \eqref{eq:twoScal}, 
\begin{align}
\Big|\frac{u(x_1,\theta,t)}{u(x_1,\theta,t_1)}-1\Big|<\epsilon.
\end{align}

\end{lem}
The intuitive idea behind Lemma \ref{keylemma} is not difficult. Recall that $x=0$ is the only blowup point, and the blowup time is $t=0$. Thus, at least intuitively, for any $x_1\not=0$, there exists a time $t_1<0$ such that for any $t\in [t_1,0)$, $u(x_1,\theta,t)$ stays away from $0$ and varies very little, and hence $\frac{u(x_1,\theta,t)}{u(x_1,\theta,t_1)}\approx 1.$ Moreover $\frac{u(x_1,\theta,t)}{u(x_1,\theta,t_1)}\rightarrow 1$ as $t\geq t_1\rightarrow 0.$

We start with reformulating Lemma \ref{keylemma} into Proposition \ref{prop:hh} below. 

Since $\displaystyle\lim_{|x|,|t|\rightarrow 0^{+}}u(x,\theta,t)\rightarrow 0$, to prove Lemma \ref{keylemma} we have to compare two  small quantities, whose derivatives are adversely large. To make our proof transparent, we rescale the MCF to get a new flow \eqref{eq:hMCF}, such that the rescaled version of $u(x_1, \theta, t_1)$ is approximately $1.$ 

Define a new MCF $\{M_{1,s}\}$ with $s\in [0,\tau_1^{-\frac{1}{10}})$, in terms of MCF $\{M_{t}\}$ by
\begin{align}\label{eq:hMCF}
M_{1,s}:=\frac{1}{\sqrt{\tau_1^\frac{1}{10}\cdot(-t_1)}}M_{t_1+s\tau_1^\frac{1}{10}\cdot(-t_1)}
\end{align}
by rescaling, and translation, in time $t$. Here $\tau_1$ chosen from \eqref{eq:twoScal} and the new time variable $s$ is defined by the following identity
\begin{align}
t=t_1+s\tau_1^\frac{1}{10}\cdot(-t_1).\label{eq:defst}
\end{align}

For the new MCF $\{M_{1,s}\}$, it blows up at time $s=\tau_1^{-\frac{1}{10}}$ by rescaling.
We also record that the $s=0, s=-1$ time slices of the new flow are given by
\begin{align}\label{TimeSlicesRescaledFlow}
M_{1,0}&=\frac{1}{\sqrt{\tau_1^\frac{1}{10}\cdot(-t_1)}}M_{t_1},\\\nonumber
M_{1,-1}&=\frac{1}{\sqrt{\tau_1^\frac{1}{10}\cdot(-t_1)}}M_{t_1+\tau_1^\frac{1}{10}\cdot t_1}.
\end{align} Recall that MCF $\{M_t\}$ blows up at the space-time origin $(0,0)$; $t_1<0$ is a fixed time very close to the blowup time 0, thus $|t_1|\ll 1$; and $\tau_1=-\ln(-t_1)$ is the corresponding time for the rescaled MCF.

The new MCF $M_{1,s}$ can be locally parametrized similarly to the original $M_{t}$,
\begin{align}\label{eq:paraM1s}
M_{1,s}=\begin{bmatrix}z\\ h(z,\theta,s)\cos\theta \\ h(z,\theta,s)\sin\theta \end{bmatrix}
\end{align} where the variable $z$ is defined as
\begin{align}\label{eq:defZH}
z:=\frac{1}{\sqrt{\tau_1^\frac{1}{10}\cdot(-t_1)}}x,
\end{align} and $h$ is a function obtained by rescaling $u,$
\begin{align}\label{eq:relahu}
h(z,\theta,s):=\frac{1}{\sqrt{\tau_1^\frac{1}{10}\cdot(-t_1)}}u(x,\theta,t).
\end{align}

The corresponding part for $x_1\in \mathbb{R}^3$ in \eqref{eq:defX1}, denoted by $z_1$, is 
\begin{align}\label{eq:defZ1}
z_1:=\frac{1}{\sqrt{\tau_1^\frac{1}{10}\cdot(-t_1)}}x_1,
\end{align}
whose length satisfies
\begin{align}\label{eq:lengZ1}
|z_1|=\frac{1}{\sqrt{\tau_1^\frac{1}{10}\cdot(-t_1)}}\cdot\sqrt{-t_1}\tau_1^{\frac{1}{2}+\frac{1}{20}}
=\tau_1^{\frac{1}{2}}.
\end{align}

The identities in \eqref{eq:defst} and \eqref{eq:defZH} -\eqref{eq:defZ1} imply that, 
\begin{align}
\frac{u(x_1,\theta,t)}{u(x_1,\theta,t_1)}=\frac{h(z_1,\theta,s)}{h(z_1,\theta,0)}.
\end{align} Consequently Lemma \ref{keylemma} is equivalent to the following result: recall that $M_{1,s}$, $s\in [0,\ \tau_1^{-\frac{1}{10}}),$ blows up at the time $s=\tau_1^{-\frac{1}{10}}\ll 1.$
\begin{prop}\label{prop:hh} For any $s\in [0,\tau_1^{-\frac{1}{10}})$,
\begin{align}
\frac{h(z_1,\theta,s)}{h(z_1,\theta,0)}=1+o(1),\label{eq:newRa}
\end{align} where $\|o(1)\|_{L^{\infty}_{\theta,s}}\rightarrow 0$ as $\tau_1\rightarrow \infty$.
\end{prop}

\begin{proof}
Many estimates for $h$ are derived from those for $v$ through the following identity, implied by the identities \eqref{eq:relahu} and \eqref{eq:uvRe},
\begin{align}
h(z,\theta,s)=&\frac{\sqrt{-t}}{\sqrt{\tau_1^\frac{1}{10}\cdot(-t_1)}} v(y,\theta,\tau)\nonumber\\
=&\frac{\sqrt{-t}}{\sqrt{\tau_1^\frac{1}{10}\cdot(-t_1)}}v(\frac{x}{\sqrt{-t}},\theta,-ln(-t))\nonumber\\
=&\frac{\sqrt{1-\tau_1^{\frac{1}{10}}s}}{\tau_{1}^{\frac{1}{20}}}\ v\big(\frac{\tau_1^{\frac{1}{20}} z}{\sqrt{1-\tau_1^{\frac{1}{10}}s}},\theta, -\ln(1-\tau_1^{\frac{1}{10}}s)+\tau_1\big),\label{eq:forHz}
\end{align} where in the last step we used the definitions of $z$ and $s$ in \eqref{eq:defZH} and \eqref{eq:defst} respectively.

The main tool in proving the desired result is pseudo-locality. To make it applicable we extend the definition of $h$ from $s\in [0,\tau_1^{-\frac{1}{10}})$ to $s\in [-1,\tau_{1}^{-\frac{1}{10}});$ and when $s\in [-1,0]$, the estimates for $v$ provide the needed ones for $h$ through \eqref{eq:forHz}. The reason we can extend to $s=-1$ is that by \eqref{TimeSlicesRescaledFlow}, the time $s=-1$ in the rescaled flow $M_{1,s}$ corresponds to time $t_1+\tau_1^\frac{1}{10} t_1$ in the original flow. Since $\tau_1=-\ln(-t_1)$, we have $t_1+\tau_1^\frac{1}{10} t_1\rightarrow0$ as $t_1\rightarrow0$. Thus we may assume without loss of generality that the asymptotic in Theorem \ref{GZthm} is valid for $v$ and thus $h$ at $s=-1$ when we have small enough $|t_1|$ so that 
\begin{align}
\bar\tau_1:=-\ln\left[-(t_1+\tau_1^\frac{1}{10} t_1)\right] >\tau_0.
\end{align}
Moreover,  for any $z$ such that $\Big||z|-|z_1|\Big|\leq 2$, one has by \eqref{eq:defZH} and \eqref{eq:lengZ1} that its corresponding length in the rescaled flow $\bar M_\tau$ (that Theorem \eqref{GZthm} applies) is
\begin{align}
|y|=\frac{|x|}{\sqrt{-t_1}}=\frac{|z|\sqrt{\tau_1^\frac{1}{10}\cdot(-t_1)}}{\sqrt{-t_1}}\leq&\frac{(|z_1|+2)\sqrt{\tau_1^\frac{1}{10}\cdot(-t_1)}}{\sqrt{-t_1}}\\\nonumber
=&\frac{(\tau_1^\frac{1}{2}+2)\sqrt{\tau_1^\frac{1}{10}\cdot(-t_1)}}{\sqrt{-t_1}}\\\nonumber
=&\sqrt{(\tau_1^\frac{1}{2}+2)^2\cdot\tau_1^\frac{1}{10}}\\\nonumber
\ll&\sqrt{2\tau_1\cdot\tau_1^\frac{1}{10}}\\\nonumber
\leq&\sqrt{100\ln(\tau_1)+9(\tau_1-\tau_0)^\frac{11}{10}}\\\nonumber
=&\Omega(\tau_1),
\end{align}
when $|t_1|$ is small enough (or in other words $\tau_1$ is large enough).

So, we can apply the asymptotic \eqref{asym} in Theorem \ref{GZthm} (see also \eqref{eq:mainP}), which gives that for $s\in [-1,0]$,
\begin{equation}
\begin{split}\label{eq:hradi}
h(z,\theta,s)&=\frac{\sqrt{1-\tau_1^{\frac{1}{10}}s}}{\tau_1^{\frac{1}{20}}}\sqrt{2+\frac{1}{\tau_1}\cdot\Big|\frac{\tau_1^{\frac{1}{20}} \ z_1}{\sqrt{1-\tau_1^{\frac{1}{10}}s}}\Big|^2}\big(1+o(1)\big)\\
&=\frac{\sqrt{1-\tau_1^{\frac{1}{10}}s}}{\tau_1^{\frac{1}{20}}}\sqrt{2+\frac{1}{\tau_1}\cdot \Big|\frac{\tau_1^{\frac{1}{20}} \ \tau_1^{\frac{1}{2}}}{\sqrt{1-\tau_1^{\frac{1}{10}}s}}\Big|^2}\big(1+o(1)\big)\\
&=\frac{\sqrt{1-\tau_1^{\frac{1}{10}}s}}{\tau_1^{\frac{1}{20}}}\sqrt{2 +\frac{\tau_1^{\frac{1}{10}} }{1-\tau_1^{\frac{1}{10}}s}}\big(1+o(1)\big)\\
&=\sqrt{1+\frac{2(1-\tau_1^\frac{1}{10}s)}{\tau_1^\frac{1}{10}}}\big(1+o(1)\big)
\end{split}
\end{equation}  
where we used that $|z_1|\gg 1$ by \eqref{eq:lengZ1}, and thus $|z|=|z_1|(1+o(1))$ when $\big||z|-|z_1|\big|\leq 2$.

When $s=0$ and $\Big||z|-|z_1|\Big|\leq 2$,
\begin{align}
h(z,\theta,0)=1+o(1)\label{eq:h0}.
\end{align} 
And when $s\in [-1,0]$, we get $h\rightarrow\sqrt{1-2s}\in[1,\sqrt3]$ as $\tau_1\rightarrow\infty$ and thus
\begin{align}
\frac{1}{2}\leq h(z,\theta,s)\leq 9
\end{align}
for large enough $\tau_1$.

To control the derivatives of $h$, we observe that, 
\begin{align}
|\nabla_{z}^{l}\partial_{\theta}^{k}h|\ll 1,\ k+|l|=1,2,\label{eq:devh}
\end{align} since the decomposition of $v$ in \eqref{asym}, the estimates on the remainder in \eqref{eq:InftyEst} and the estimates on the scalar functions in \eqref{eq:b123} and \eqref{eq:aaa} imply that
\begin{align}\label{DerivativeBound}
|\nabla_{y}^{l}\partial_{\theta}^{k}v(\cdot,\tau)|\leq \tau^{-\frac{3}{20}},\  k+|l|=1,2.
\end{align}
Higher order derivatives follows by \cite{Ec}, Proposition 2.17.

Recall that, by the parametrization of the hypersurface $\{M_{1,s}\}$ in \eqref{eq:paraM1s}, $h(z,\theta,s)$ can be interpreted as the radius of the hypersurface at the cross section $z$ and at the angle $\theta$. Thus the estimates \eqref{eq:hradi} and \eqref{eq:devh} above imply that, when $s\in [-1,0]$ and $\Big||z_1|-|z|\Big|\leq 2$, the hypersurface is close to a cylinder.

We are ready to use the pseudo-locality property (see Theorem 1.5 of \cite{INS}) as \eqref{DerivativeBound} gives the Lipschitz bound in the conditions: when $s\in[0,\tau_1^{-\frac{1}{10}}]$, which is a small interval since $\tau_1\gg 1$, the flow can be parametrized in $B_{\delta_0}(z_1)$ as a graph $h$ for some uniform $\delta_0>0$ and satisfies
\begin{align}
h(z_1,\theta,s)=h(z_1,\theta,0)\Big(1+o(1)\Big).
\end{align}

Thus we have the desired estimate $$\frac{h(z_1,\theta,s)}{h(z_1,\theta,0)}=1+o(1).$$ 
\end{proof}

\section{Proof of the main theorem}
Now we are ready to prove the main theorem.
\begin{proof}[Proof of the Main Theorem \ref{main}]
We start with simplifying the problem. Lemma \ref{keylemma} shows that, for each fixed $x_1$ with $0<|x_1|\ll 1$, for $t\in [t_1,0)$, with $t_1=t_1(x_1)$ defined by \eqref{eq:defX1} and \eqref{eq:twoScal},
\begin{align}
u(x_1,\theta,0)=&u(x_1,\theta,t)\Big(1+o(1)\Big).\label{eq:relation}
\end{align}

Hence, after passing to the limit $t\rightarrow 0$, to control $u(x_1,\theta,0)$ it suffices to estimate $u(x_1,\theta,t_1).$

To estimate $u(x_1,\theta,t_1)$ we need estimates for $v$ introduced in \eqref{eq:uvRe}, which appears in the rescaled MCF. %The facts $$\tau_1=-\ln |t_1|,\ \  y_1=\frac{x_1}{\sqrt{-t_1}}$$ and the definition relating $u$ and $v$ in \eqref{eq:uvRe} make
%\begin{align}\label{eq:ux1vy1}
%u(x_1,\theta,t_1)=\sqrt{-t_1}v(\frac{1}{\sqrt{-t_1}}x_1,\theta,\tau_1)
%=\sqrt{-t_1}v(y_1,\theta,\tau_1).
%\end{align}

To control $v$ we apply Theorem \ref{GZthm}.
A useful estimate is that: when $|y|\leq 2\tau^{\frac{1}{2}+\frac{1}{20}},$
\begin{align}
v(y,\theta,\tau)=\sqrt{2+\tau^{-1}|y|^2}\big(1+o(1)\big) ,\label{eq:mainP}
\end{align} where $o(1)$ is in the $L_{y,\theta}^{\infty}$-norm. To see this, $v$ is decomposed into two parts
\begin{align}
v(y,\theta,\tau)=\sqrt{\frac{2+y^T\cdot B(\tau)\cdot y}{2a(\tau)}}+\eta(y,\theta,\tau).
\end{align} The first estimate in \eqref{eq:InftyEst} implies that, when $|y|\leq 3\tau^{\frac{1}{2}+\frac{1}{20}}$ and $\tau\gg 1$, then 
\begin{align}
|\eta(y,\theta,\tau)|\leq \left(\sqrt{1+y^2}\right)^{3}\tau^{-2} \leq \tau^{-\frac{1}{10}}.
\end{align} This, together with $B=\tau^{-1}I+O(\tau^{-2})$, $a=\frac{1}{2}+O(\tau^{-1})$ in \eqref{eq:b123} and \eqref{eq:aaa}, implies the desired \eqref{eq:mainP}.

Now since $|y_1|=e^\frac{\tau_1}{2}|x_1|= \tau_1^{\frac{1}{2}+\frac{1}{20}}$ by \eqref{eq:defX1} and $\sqrt{-t_1}=e^{-\frac{\tau_1}{2}}$, we apply the estimate in \eqref{eq:mainP} to obtain,
\begin{align}
u(x_1,\theta,t_1)=&e^{\frac{-\tau_1}{2}}\sqrt{1+\frac{1}{\tau}|y_1|^2}(1+o(1))\nonumber\\
=&e^{\frac{-\tau_1}{2}}\sqrt{1+\frac{1}{\tau_1}(\tau_1^{\frac{1}{2}+\frac{1}{20}})^2}(1+o(1))\nonumber\\
=&e^{\frac{-\tau_1}{2}}\sqrt{1+\tau_1^{\frac{1}{10}}}(1+o(1))\nonumber\\
=&e^{\frac{-\tau_1}{2}}\tau_1^\frac{1}{20}(1+o(1)).\label{eq:ux1t1}
\end{align}

%This estimate for $u(x_1,\theta,t_1)$ is in terms of $\tau_1$ while the desired is in terms of $|x_1|$. 
We convert it using the identity $|x_1|=\tau_1^{\frac{1}{2}+\frac{1}{20}}e^\frac{-\tau_1}{2}$ from \eqref{eq:defX1} so that the estimate is in terms of $|x_1|$. Compute directly to find 
\begin{align}
u(x_1,\theta,t_1)=&
\frac{|x_1|}{\tau_1^{\frac{1}{2}}}\big(1+o(1)\big)=\frac{|x_1|}{\sqrt{-2\ln|x_1|}}\big(1+o(1)\big),
\end{align} where, in the last step, we took a logarithm on both sides of $|x_1|=\tau_1^{\frac{1}{2}+\frac{1}{20}}e^\frac{-\tau_1}{2}$ to find
\begin{align*}
-\ln|x_1|= \frac{\tau_1}{2}-\frac{11}{20}\ln \tau_1=\frac{\tau_1}{2}\big(1+o(1)\big),
\end{align*} since $\tau_1\gg 1.$

This, together with \eqref{eq:relation}, implies the desired results.
\end{proof}

\end{document}